\documentclass[12pt]{amsart}
\usepackage{amsmath, amsthm, amssymb, mathtools, color}
\usepackage[shortlabels]{enumitem}
\usepackage{epsfig, mathrsfs}
\usepackage{latexsym}
\usepackage{pstricks}
\usepackage{pgf,tikz}
\usetikzlibrary{arrows}
\usepackage{subcaption}

\usepackage{pdfpages}
\usepackage{tikz}
\usepackage{epstopdf}
\usepackage{float}
                                                                                          
\newtheorem{theorem}{Theorem}
\numberwithin{theorem}{section}
\newtheorem{proposition}[theorem]{Proposition}
\newtheorem{lemma}[theorem]{Lemma}

\newtheorem{remark}[theorem]{Remark}
\newtheorem{example}[theorem]{Example}

\newcommand{\R}{\mathbb{R}}
\newcommand{\Z}{\mathbb{Z}}

\newcommand{\C}{\mathscr{C}}
\newcommand{\SSS}{\mathscr{S}}
\newcommand{\A}{\mathscr{A}}
  \date{}
 
\title{Symmetric matrices, Catalan paths, and correlations}
\author{Bernd Sturmfels}
\address{}
\email{}
\author{Emmanuel Tsukerman}
\address{}
\email{}
\author{Lauren Williams}
\address{}
\email{}

\begin{document}

\begin{abstract}
Kenyon and Pemantle (2014) gave a
formula for the entries of a square matrix in terms
of connected principal and almost-principal minors.
Each entry is an explicit Laurent polynomial whose terms 
are the weights of
 domino tilings of a half Aztec diamond.    They conjectured an
analogue of this parametrization for symmetric matrices,
where the Laurent monomials are indexed by Catalan paths.
In this paper we prove the Kenyon-Pemantle conjecture,
and apply this to a statistics problem pioneered by Joe~(2006).
Correlation matrices are represented by  an explicit
bijection from the cube to the elliptope.
\end{abstract}
\maketitle
%\setcounter{tocdepth}{1}
%\tableofcontents

\section{Introduction}

In this  paper we present a formula for each entry of a 
symmetric $n \times n$ matrix $X = (x_{ij})$
as a Laurent polynomial in $\binom{n+1}{2}$
distinguished minors of $X$. Our result verifies a conjecture of
 Kenyon and Pemantle from \cite{KP2}. 
Let $I$ and $J$ be subsets of $[n] = \{1,2,\dots,n\}$ with
$|I| = |J|$.  
Let $X^J_I$ denote  the minor of $X$ with row indices $I$ and column indices $J$.
Here the indices in $I$ and $J$ are always taken in increasing order.
The following signed minors will be used:
$$ \,\,
\begin{matrix} 
& p_I  & := &  (-1)^{\lfloor\, |I|/2\, \rfloor} \cdot X_I^{I}  & \\
{\rm and} & a_{ij|I} & := &  (-1)^{\lceil\, |I|/2 \,\rceil} \cdot X^{jI}_{iI} & \,\,\,
\hbox{for} \,\,\, i,j \not \in I, \quad i \neq j.
\end{matrix}
$$
We call $p_I$ and $a_{ij|I}$ the \emph{principal}
and \emph{almost-principal} minors, respectively.
The minors $p_I$, $a_{ij|I}$ and $a_{ji|I}$ are called \emph{connected}
if $1\leq i<j \leq n$ and $I = \{i{+}1,i{+}2,\dots,j{-}2,j{-}1\}$.
The $1 {\times} 1$-minors $a_{ij} := a_{ij|\emptyset} = x_{ij}$ 
and $p_k = x_{kk}$ are connected when
$|i-j|=1$ and $ 1 \leq k \leq n$.

These definitions make sense for every $n \times n$ matrix $X$,
even if $X$ is not symmetric.
A general $n {\times} n$ matrix $X$ has $2^n$ principal minors, 
of which $\binom{n-2}{2}+n$ are connected. It also has
$n(n-1)2^{n-2}$ almost-principal minors, of which $n(n-1)$ are connected.
A symmetric $n {\times} n$ matrix has
$\binom{n}{2} 2^{n-2}$ distinct almost-principal minors
$a_{ij|I}$, of which $\binom{n}{2}$ are connected.

A \emph{Catalan path} $C$ is
a path in the $xy$-plane which starts at $(0,0)$ and ends on the $x$-axis,
always stays at or above the $x$-axis, and 
consists of steps
northeast $(1,1)$ and southeast $(1,-1)$.
We say that $C$ has \emph{size $n$} if its endpoints 
have distance $2n-2$ from each other.
Let $\C_n$ denote the set of Catalan paths of size $n$.
Its cardinality equals the Catalan number
$$ \qquad \quad |\C_n| \,= \, \frac{1}{n} \binom{2n-2}{n-1},
\quad \hbox{which is} \,\,
1, 2, 5, 14, 42, 132, 429, 1430, 4862 \,\,
\hbox{for $n=2,\ldots,10$}.
 $$
 
 Let $G_n$ denote the planar graph whose vertices 
  are the $\binom{n+1}{2}$ lattice points $(x,y)$ with
 $x \geq y \geq 0$ and $ x+y \leq 2n-2$ even,
 and  edges are northeast and southeast steps.
 Thus $\C_n$ consists of the paths from $(0,0)$ to $(2n-2,0)$ in $G_n$.
 We label the nodes and regions of $G_n$ as follows.
We assign label $j$ to the node $(2j-2,0)$,  label $a_{ij|I}$ to the node $(i+j-2,j-i)$,
and label $p_I$ to the region  below that node. 
 Thus,  connected principal and almost-principal minors of $X$
are identified in the graph $G_n$
with  regions  and  nodes strictly above the $x$-axis.
 
The \emph{weight} $W_\C(C)$ of a Catalan path $C$
is a Laurent monomial, derived from the drawing of $C$ 
 in the graph ${G}_n$. Its numerator is the product
of the labels $a_{ij|I}$ of the nodes  of $G_n$ that are local maxima
or local minima of $C$,
and its denominator is the product of the labels $p_I$ of the regions 
which are either immediately 
below a local maximum or immediately above a local minimum. 
Thus $W_\C(C)$ is a Laurent monomial of degree~$\leq 1$.
There is no lower bound on the degree;
for instance, $\frac{a_{13|2} a_{35|4} a_{57|6} a_{79|8}}{
p_2 p_3 p_4 p_5 p_6 p_7 p_8}$ has degree $-3$ and
appears for $ n = 9$.

\smallskip

The following result was
conjectured by Kenyon and Pemantle in
\cite[Conjecture 1]{KP2}.

\begin{theorem} \label{thm:main}
The entries of an $n \times n$ symmetric matrix $X = (x_{ij})$ satisfy the identity
\begin{equation}
\label{toprove}
 x_{ij} \,\,= \,\,\sum_C W_\C(C),
 \end{equation}
where the sum is over all Catalan paths $C$ between
node $i$ and node $j$ in $G_n$
\end{theorem}

\begin{figure}[h]
\centering
\includegraphics[height=6cm]{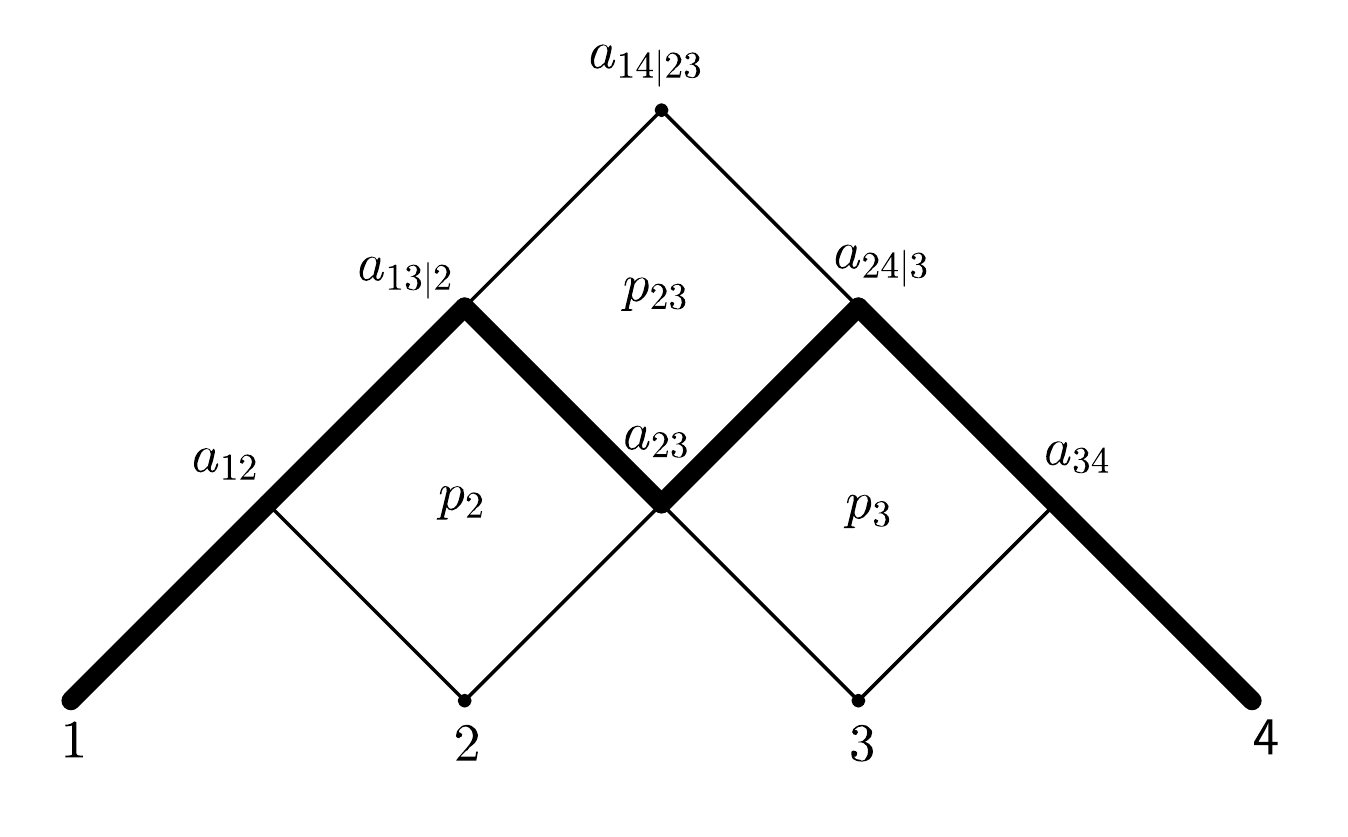}
\vspace{-0.25in}
\caption{A Catalan path $C$  in the planar graph ${G}_4$
with weight
$ \frac{a_{13|2} a_{23} a_{24|3}}{p_2 p_{23} p_3 }$.
} 
\label{fig:Catalan}
\end{figure}

For symmetric matrices of size $n = 4$, Theorem~\ref{thm:main} states the 
following formula:
\begin{equation}
\label{eq:Xmatrix}
%$$
 X \, = \,
\begin{pmatrix}\,
p_1\, & \,a_{12}\, & \frac{a_{13|2}}{p_2} + \frac{a_{12}a_{23} \phantom{|}}{p_2} &
 \frac{a_{14|23}}{p_{23}} + \frac{a_{12} a_{24|3}}{p_2 p_3} + \frac{
 a_{13|2} a_{34}}{p_2 p_3} + \frac{a_{12} a_{23} a_{34}\phantom{|}}{p_2 p_3} 
 + \frac{a_{13|2} a_{23} a_{24|3} }{
 p_2 p_{23} p_3}  \smallskip \\
\,* & p_2 & a_{23} & \frac{a_{24|3}}{p_3} + \frac{a_{23} a_{34} \phantom{|}}{p_3} \smallskip \\
\,*& * & p_3 & a_{34} \smallskip \\
\,*  & * & * & p_4
\end{pmatrix}
%$$
\end{equation}
The entry $x_{14}  = x_{41}$ is the sum of 
five Laurent monomials, one for each Catalan path 
from node $1$ to node $4$.
The last term 
$ \frac{a_{13|2} a_{23} a_{24|3}}{p_2  p_{23}p_3}$
equals $W_\C(C)$ for the path $C$
shown in Figure~\ref{fig:Catalan}. 

\smallskip

The proof of Theorem \ref{thm:main} 
%BS will be
is given in Section \ref{sec:proof}. 
We start in Section \ref{az} by reviewing
a theorem of Kenyon and Pemantle \cite{KP2} 
which expresses the entries of an arbitrary square matrix in terms of almost-principal and principal minors,
as a sum of Laurent monomials that are in bijection with domino tilings of a 
half Aztec diamond.  In Section \ref{sec:Schroder}, 
we give a bijection between these domino tilings and 
Schr\"oder paths, and restate their theorem using Schr\"oder paths.
We then prove our theorem
%BS in Section \ref{sec:proof}, 
by constructing a projection from 
Schr\"oder paths to Catalan paths and applying the relation \eqref{eq:relation}
%BS
among minors of symmetric matrices.

In Section  \ref{sec5} we connect Theorem \ref{thm:main}
to an application in statistics, developed in work of
Joe, Kurowicka and Lewandowski \cite{MR2301633, MR2543081}.
Namely, we focus on symmetric matrices that are positive
definite and have all diagonal entries equal to $1$.
These are the {\em correlation matrices}, and they
form a convex set that is  known in optimization as the {\em elliptope} \cite{frgbook, LP}.
Our formula  yields an explicit bijection between
the elliptope and  the open cube $(-1,1)^{\binom{n}{2}}$.

\section{Square matrices and tilings of the half Aztec diamond \label{az}}

In this section we review the  Kenyon-Pemantle formula in \cite[Theorem 4.4]{KP2}.
The \emph{half Aztec diamond} $HD_n$ of order $n$ is the
union  of the unit squares whose vertices are in the~set
$$\{\,(a,b) \in \Z^2 \,\,:  \,\,  \, |a| \leq n, \,0 \leq b \leq n,\,|a|+|b| \leq n+1\,\}. $$
We label the boxes in the bottom row of $HD_n$ by the numbers $1$ through $2n$, from 
left to right.
We label certain lattice points of $HD_n$ by minors
 as follows. Fix $b \in [n]$. 
 The connected principal minors $p_I$ such that $|I|=b$ 
 % $1 \notin I$, and $n\notin I$
 are assigned to the lattice points $(a,b)$ with $a+b$ even.
 The connected almost-principal minors $a_{ij|I}$ with $i > j$ and $|I|=b-1$
  are assigned to the lattice points $(a,b)$ with $a+b$ odd.
  In both cases, the assignment is from left to right using
   the lexicographic order on  $I$.
   The case $n=4$ is shown in Figure \ref{fig:hd4}.

\begin{figure}[h]
\centering
\includegraphics[height=6.2cm]{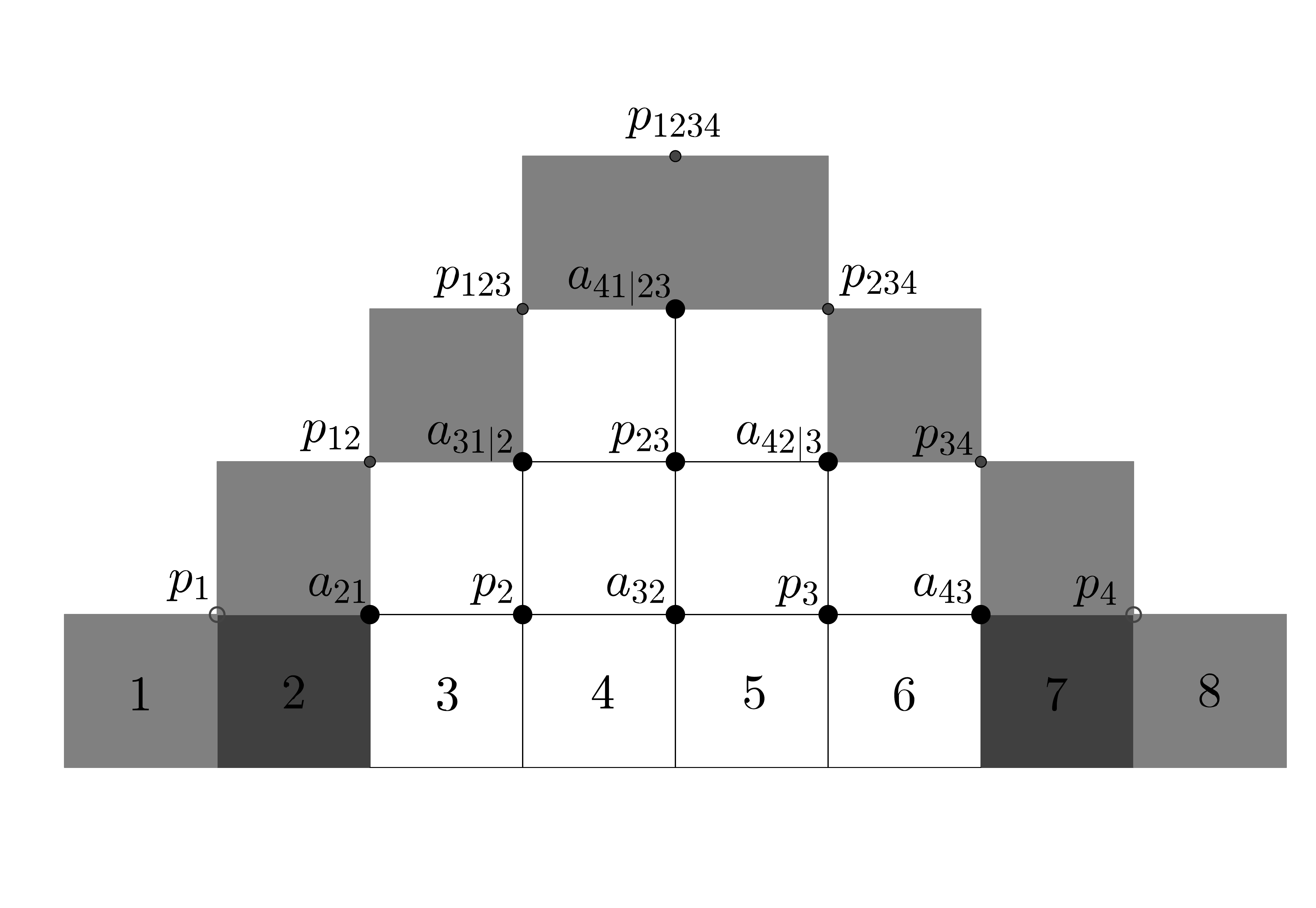}
\vspace{-0.2in}
\caption{The half Aztec diamond $HD_4$. 
The white boxes are to be tiled.}
\label{fig:hd4}
\end{figure}

Fix  integers $a$ and $b$ such that $a$ is even, $b$ is odd, and 
$1 {<} a {<} b {<} 2n$. We define the \emph{colored half Aztec diamond} $HD_n(a,b)$ 
by coloring the boxes of $HD_n$ black, grey, or white.
First color boxes $a$ and $b$ in the bottom row black.  
Let $L_a$ be the diagonal line of slope $1$ through box $a-1$,
and let $L_b$ be the line of slope $-1$ through box $b+1$.
If a box (or any part of it) lies to the left of $L_a$ or to the right of $L_b$, then color it grey.
All other boxes are %colored 
white.
A \emph{domino 
tiling} (or simply a \emph{tiling}) 
of $HD_n(a,b)$ is a tiling of the white boxes  by
$1 {\times} 2$ and $2 {\times} 1$ rectangles.
Let $\A_n(a,b)$ denote the set of tilings of $HD_n(a,b)$.
Figure \ref{fig:aztecTiling2} 
shows the set $\A_4(2,7)$, i.e.~the six tilings of $HD_4(2,7)$, with lines 
$L_2$ and $L_7$ superimposed on the tilings.

\begin{figure}[h]
\centering
\includegraphics[width=0.90\textwidth]{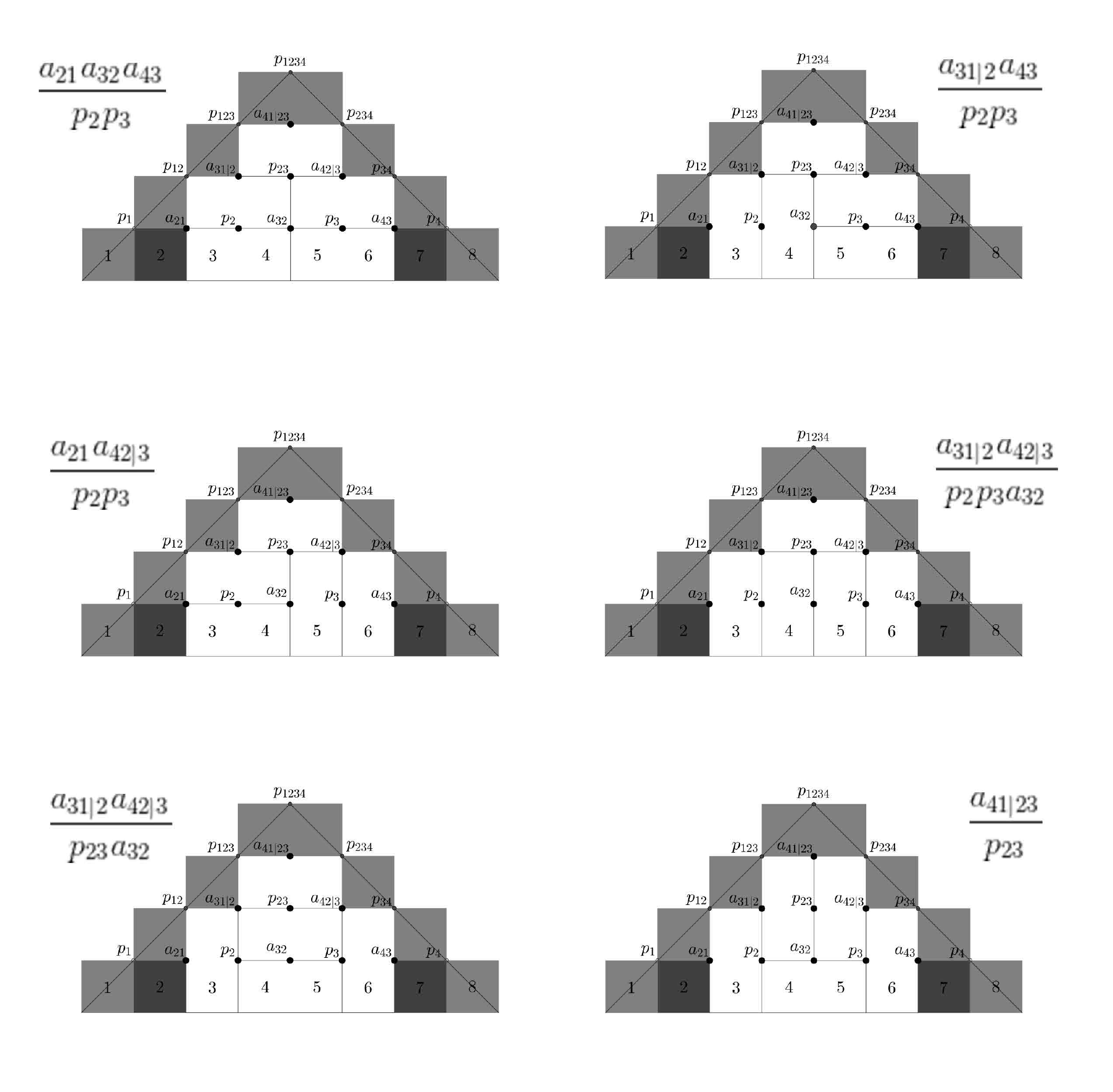}
\vspace{-0.38in}
\caption{
\label{fig:aztecTiling2}
The six tilings of 
the colored half Aztec diamond $HD_4(2,7)$.}
\end{figure}

Each tiling $T$ of the colored half Aztec diamond
$HD_n(a,b)$ gets a Laurent monomial weight, which we now define.
We regard $T$ as a simple graph whose nodes
are the lattice points of 
$HD_n$, and whose edges are induced 
by the edges of the rectangles in the tiling together with the 
edges of the unit squares outside the tiling.
An \emph{interior lattice point} of $HD_n(a,b)$ is a lattice point which lies strictly 
to the right of $L_a$ and strictly to the left of $L_b$.  The interior lattice points
that will concern us are shown in bold in Figures 
\ref{fig:hd4}
and \ref{fig:aztecTiling2}. 
 Each of these is %either unlabeled, 
labeled by  a variable
$v_{\ell}$ which is a connected principal or almost-principal minor.
The \emph{weight} $W_\A(T)$ of a tiling $T\in \A_n(a,b)$ is defined to be
the Laurent monomial
$$W_\A(T) \,\,:= \,\, \prod_\ell v_{\ell}^{d(\ell)-3},$$
where $\ell$ ranges over the interior lattice points of $HD_n(a,b)$ and 
$d(\ell)$ is the degree of $\ell$ in $T$. 

\begin{theorem}[Kenyon-Pemantle \cite{KP2}] \label{KP-theorem}
The entries of an $n \times n$ matrix $X = (x_{ij})$ satisfy
$$ x_{ij}  \quad = \sum_{T\in \A_n(2j, 2i-1)} W_\A(T)
\qquad {\rm for} \,\, i > j .$$
\end{theorem}

Theorem 4.4 in \cite{KP2} also gives a similar formula for 
$x_{ij}$ with  $i<j$, but we omit that formula, as it is not needed here.

\begin{example}\label{ex:4by4} \rm
Figure \ref{fig:aztecTiling2} shows the six tilings of $HD_4(2,7)$ with
 their weights.  By Theorem~\ref{KP-theorem},
the upper right matrix entry for $n = 4$ is
the sum of these six Laurent monomials:
\begin{equation}
\label{eq:x41}
x_{41} \,\,= \,\,\frac{a_{21} a_{32} a_{43}}{p_2 p_3} + \frac{a_{31|2} a_{43}}{p_2 p_3} +
\frac{a_{21} a_{42|3}}{p_2 p_3} + \frac{a_{31|2} a_{42|3}}{p_2 p_3 a_{32}}+
\frac{a_{31|2} a_{42|3}}{ p_{23} a_{32}} + \frac{a_{41|23}}{p_{23}}.
\end{equation}
The full $4 \times 4$ matrix is shown 
on page 8 of \cite{KP2}, albeit with different notation.
\end{example}

\section{Square matrices and Schr\"oder paths}\label{sec:Schroder}

In this section we continue our discussion
of arbitrary  square matrices.
A \emph{Schr\"oder path} $S$ is %of order $n$ is 
a path in the $xy$-plane which starts  at $(0,0)$,
always stays at or above the $x$-axis, and 
consists of steps which are either
northeast $(1,1)$, southeast $(1,-1)$, or horizontal $(2,0)$.
A Schr\"oder path has \emph{order $n$} if it ends at $(2n-4,0)$.
Let $G_n'$ denote the planar graph whose
nodes are the lattice points $(x,y)$ with 
$0 \leq  y \leq x $ and $x+y \leq 2n-4 $ even,
with edges given by northeast, southeast and horizontal steps.
The set $\SSS_n$ of Schr\"oder paths of order $n$
is identified with the left-to-right paths in $G_n'$ from $(0,0)$
to $(2n-4,0)$.
The cardinality  of $ \SSS_n$ is the {\em Schr\"oder number},
which is given by the generating function
$$ 
\sum_{n=2}^\infty |\SSS_n| z^{n-2} \,\, = \,\,
\frac{1 - z - \sqrt{ 1 - 6z + z^2}}{2z} \, = \,
 1 + 2z + 6 z^2 + 22 z^3 + 90 z^4 + 394 z^5 + 
1806 z^6 + \cdots $$

The graph $G_n'$ is labeled by connected minors.
We assign $a_{ij|I}$ to  the node $(i{+}j{-}3,i{-}j{-}1)$ for $i>j$,
and we assign  $p_I$ to the triangle below that node. 
We refer to $(2i-2,0)$ as  node $i$.
Figure \ref{fig:Schroder}  shows the case $n=4$.
The six Schr\"oder paths in $\SSS_4$
are shown in  Figure \ref{fig:Schroder2}.

\begin{figure}[h]
\centering
\vspace{-0.15in}
\includegraphics[height=6.7cm]{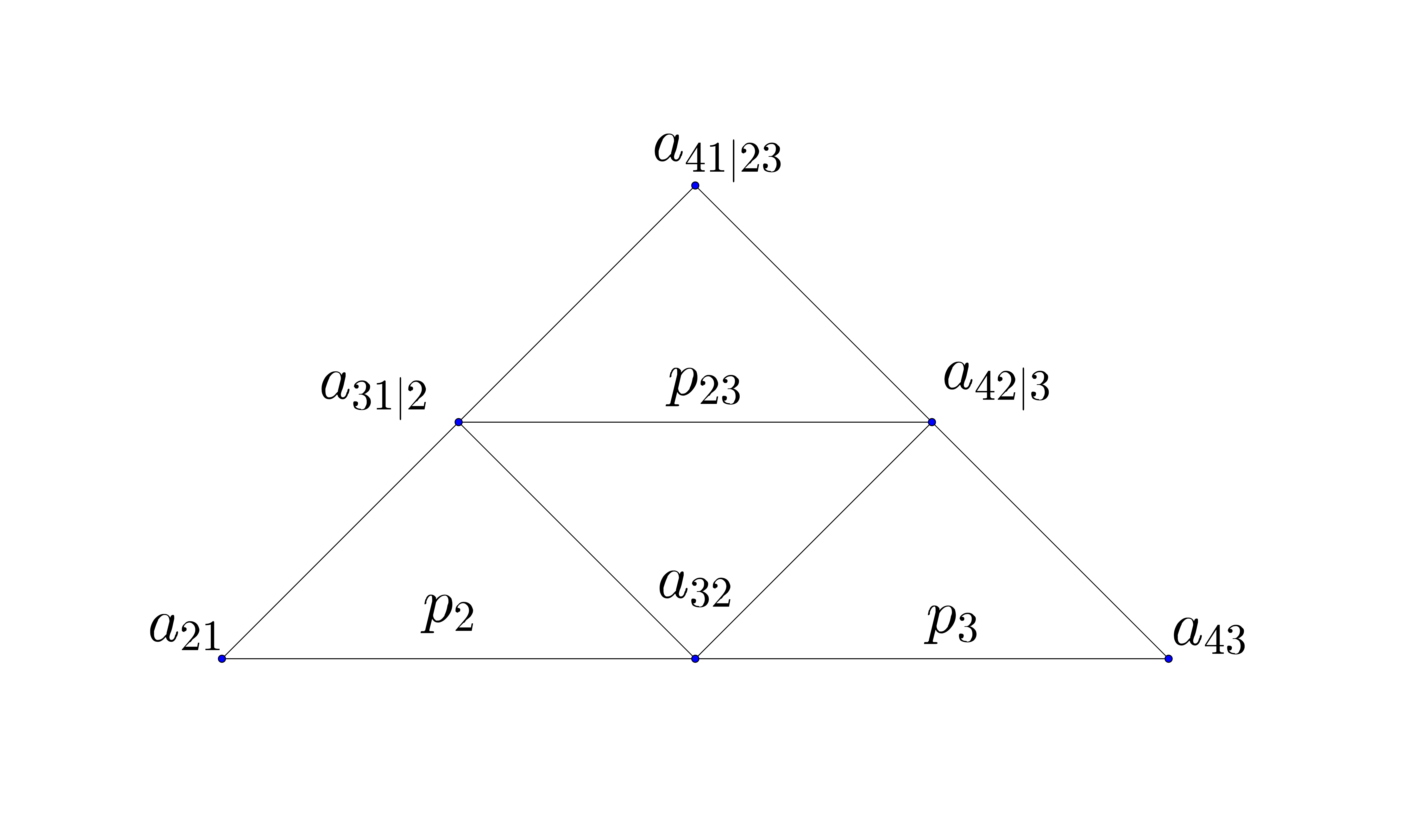}
\vspace{-0.5in}
\caption{The graph $G'_4$ encodes the Schr\"oder paths of order $4$.}
\label{fig:Schroder}
\end{figure}

We now define the \emph{weight} $W_\SSS(S)$ of a Schr\"oder path $S$ on $G'_n$.
We regard $S$ as a graph with vertices $V(S)$ and edges $E(S)$.
Given a Schr\"oder path $S$ on $G'_n$, 
we define the sets
$$ \begin{matrix}
\alpha(S)&= & \{v \in V(S) :v \text{ is a weak local maximum of }S\}, \qquad \qquad \qquad \qquad\,\, \\
\beta(S)&= & \{e \in E(S):e \text{ is immediately below a weak local minimum of }S\}, \\
\gamma(S)&= & \{e \in E(S):e \text{ is a horizontal edge of } S\},
\qquad \qquad \qquad \qquad \qquad \quad  \\
\delta(S)&= & \{v\in V(S):v \text{ is immediately below a horizontal edge of }S\}, \quad\quad\,\, \\
\epsilon(S)&= &\, \{e\in E(S) :e \text{ is immediately below a strict local maximum of }S\}, \\
\zeta(S)&= &\quad\,\,\, \{v \in V(S) :v \text{ is a strict local minimum (but not an endpoint) of }S \}.
\end{matrix} $$
Each of these is regarded as a monomial
by taking the product of all labels.
Then we define
\begin{align}\label{sweight}
W_\SSS(S)\,\,=\,\,\frac{\alpha(S)\beta(S)}{\gamma(S)\delta(S)\epsilon(S)\zeta(S)}.
\end{align}
Figure \ref{fig:Schroder2} shows 
 the six Schr\"oder paths for $n=4$, together with their weights.
  The sum of these weights is the Laurent polynomial in \eqref{eq:x41},
  which evaluates to the matrix entry $x_{41}$.
 
 The main result of this section is a reformulation of Theorem \ref{KP-theorem}
in terms of Schr\"oder paths.  
We write $\SSS_n(a,b)$ for the set of all Schr\"oder paths
from node $a$ to node $b$ in $G'_n$.

\begin{theorem}\label{thm:Schroder}
The entries of an  $n \times n$ matrix $X = (x_{ij})$ 
satisfy 
$$x_{ij} \,\,\,\, = \sum_{S \in \SSS_n(j,i-1)}  \!\! W_\SSS(S)
\qquad {\rm for} \,\,\, i > j .$$
\end{theorem}

%BS one more line saved in this paragraph
 We shall present a weight-preserving bijection
 $\Phi:\A_n(2j,2i-1) \rightarrow \SSS_n(j,i-1)$
between tilings and Schr\"oder paths.
Note that we can superimpose the graph $G'_n$ on the graph
$HD_n$ so that the labels (connected minors) match up.
When we do this, the vertex $j$ (respectively,  $i-1$) of $G'_n$ gets
identified with the top right corner
of the square $2j$ (respectively, the top left corner of the square
 $2i-1$) in $HD_n$.  We  draw a Schr\"oder path $\Phi(T)$
on top of a tiling~$T$,  as in Figure~\ref{fig:bijection}.
 We may then think of the path as 
an element of $\SSS_n(j,i-1)$.

\begin{figure}[h]
\begin{tikzpicture}
\filldraw[fill=blue!40!white, draw=black] (0,0) rectangle (2,4);
\node (b) at (0,2)[circle,fill=black]  {};
\node (t) at (2,4)[circle,fill=black]  {};
\draw[-latex] (b) -> (t);
\end{tikzpicture} 
%\]
%\[
\hspace{.5cm}
\begin{tikzpicture}
\filldraw[fill=blue!40!white, draw=black] (0,0) rectangle (2,4);
\node (b) at (0,4)[circle,fill=black]  {};
\node (t) at (2,2)[circle,fill=black]  {};
\draw[-latex] (b) -> (t);
\end{tikzpicture} 
%\]
%\[
\hspace{.5cm}
\begin{tikzpicture}
\filldraw[fill=blue!40!white, draw=black] (0,0) rectangle (4,2);
\node (b) at (0,2)[circle,fill=black]  {};
\node (t) at (4,2)[circle,fill=black]  {};
\draw[-latex] (b) -> (t);
\end{tikzpicture} 
%\]
\caption{\label{fig:bijection} How to construct a Schr\"oder path from a tiling.}
\end{figure}
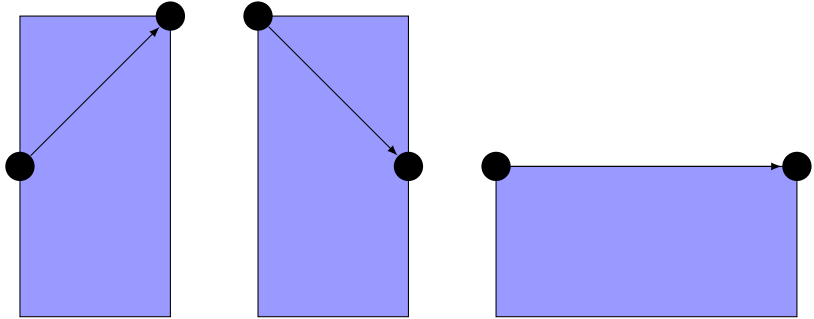

\begin{figure}[h]
\centering
\includegraphics[width=0.9 \textwidth]{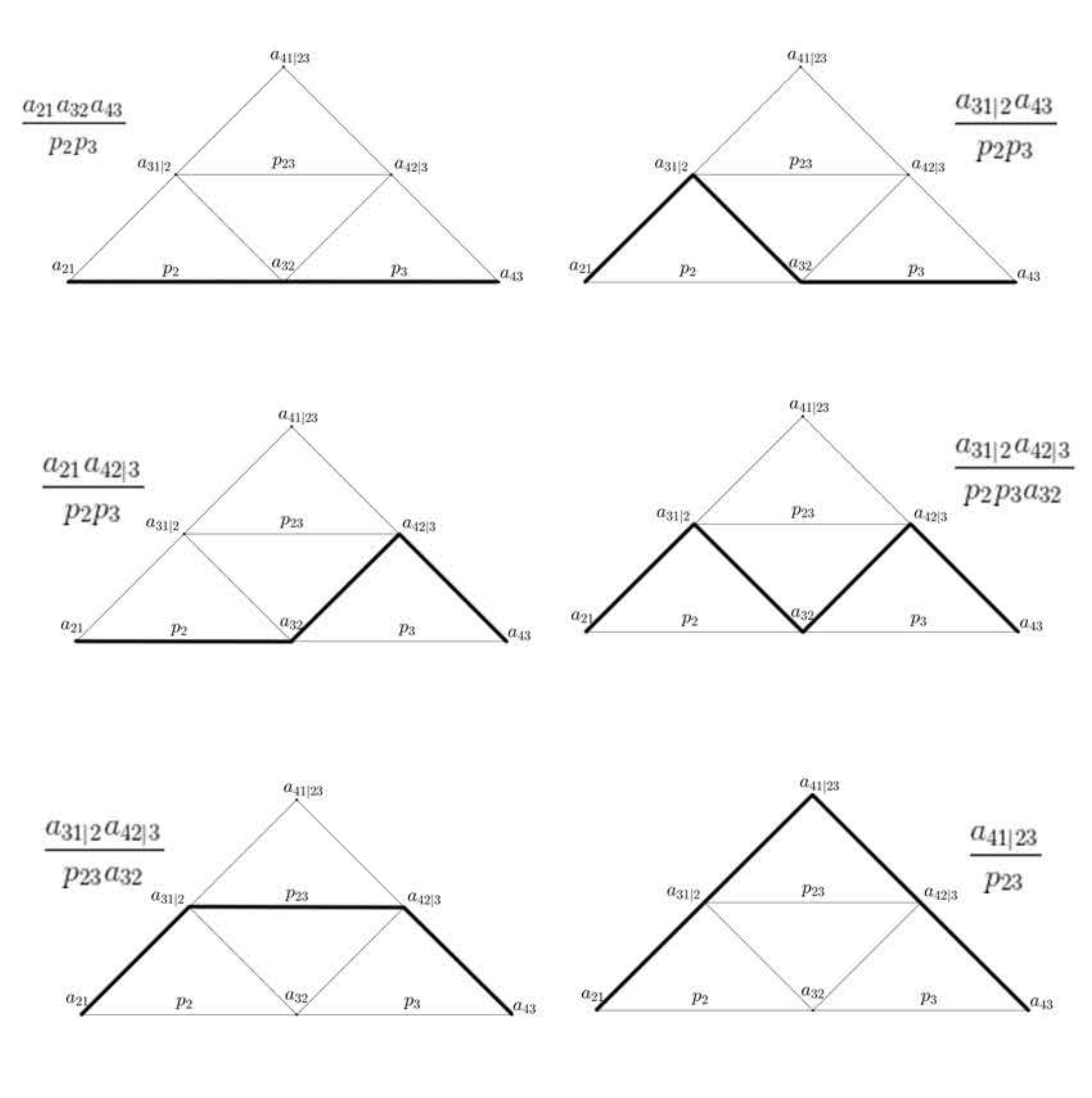}
\vspace{-0.26in}
\caption{The six Schr\"oder paths in $\SSS_4$ together with their weights.
\label{fig:Schroder2}}
\end{figure}

More formally, given $T \in \A_n(2j,2i-1)$,
the path    $\Phi(T) \in \SSS_n(j,i-1)$   is defined as follows.
Its starting point is the top right corner of square $2j$ in $HD_n(2j,2i-1)$.
We inductively add steps to $\Phi(T)$ depending on the local behavior
of the tiling, as shown in Figure~\ref{fig:bijection}.
  Let $x$ denote the endpoint of the path that we have built so far. 
  Then we proceed as follows:
\begin{itemize}
\item If there is a vertical tile to the east of $x$, 
then we add a northeast step to our path.
\item If there is a vertical tile to the southeast of $x$, such that $x$ is 
at its northwest corner, then we add a southeast step to our path.
\item If there is a horizontal tile to the southeast of $x$, then 
add an east step to our path.
\item If $x$ is already at the top left corner of square $2i-1$, then we stop.
\end{itemize}
The map $\Phi$ maps the six tilings in Figure 
\ref{fig:aztecTiling2}
to the six Schr\"oder paths in Figure 
\ref{fig:Schroder2}.

\begin{lemma}\label{bij}
The map $\Phi: \A_n(2j,2i-1) \to \SSS_n(j,i-1)$ is well-defined and is 
a bijection.
%AZ_n$ to $P_n$. The restriction $\Phi \vert _{HAZ_n}$ is a bijection from $HAZ_n$ to $S_n$.
\end{lemma}

\begin{proof}
This is the solution to Exercise 6.49 in
\cite{Stanley}, based on an idea of Dana Randall.
\end{proof}

%\begin{example}
%\begin{figure}[H]
%\centering
%\includegraphics[height=11cm]{Graphics/fullAztec}
%\caption{Example of Theorem \ref{bij}.
%\label{fig:fullAztec}}
%\end{figure}
%\end{example}

Proposition \ref{PathWeight} states that this bijection is weight-preserving. 
First, another lemma:

\begin{lemma}\label{lem:move}
The local move shown in the top of Figure \ref{fig:move1}
alters the weight of both the tiling and the corresponding Schr\"oder path
by the same factor: when passing from the left to the right, 
the exponents of $b$ and $h$   increase by $1$, 
while those of
 $d$ and $f$ decrease by~$1$. 

\begin{figure}
\begin{tikzpicture}[scale = 0.6]
\node[align=left] at (6.5,2) {$\longleftrightarrow$};
\filldraw[fill=blue!40!white, draw=black] (0,2) rectangle (4,4);
\filldraw[fill=blue!40!white, draw=black] (0,0) rectangle (4,2);
\node (b3) at (0,2)[circle,fill=black]  {};
\node (t3) at (4,2)[circle,fill=black]  {};
\draw[-latex] (b3) -> (t3);

\filldraw[fill=blue!40!white, draw=black] (11,0) rectangle (13,4);
\node (b) at (11,4)[circle,fill=black]  {};
\node (t) at (13,2)[circle,fill=black]  {};
\draw[-latex] (b) -> (t);
\filldraw[fill=blue!40!white, draw=black] (9,0) rectangle (11,4);
\node (b22) at (9,2)[circle,fill=black]  {};
\node (t2) at (11,4)[circle,fill=black]  {};
\draw[-latex] (b22) -> (t2);

\node (a1)[label=left:{$a$}] at (0,0)  {};
\node (b1)[label=left:{$b$}] at (2,0)  {};
\node (c1)[label=left:{$c$}] at (4,0)  {};
\node (d1)[label=left:{$d$}] at (0,2)  {};
\node (e1)[label=left:{$e$}] at (2,2)  {};
\node (f1)[label=left:{$f$}] at (4,2)  {};
\node (g1)[label=left:{$g$}] at (0,4)  {};
\node (h1)[label=left:{$h$}] at (2,4)  {};
\node (i1)[label=left:{$i$}] at (4,4)  {};

\node (a2)[label=left:{$a$}] at (9,0)  {};
\node (b2)[label=left:{$b$}] at (11,0)  {};
\node (c2)[label=left:{$c$}] at (13,0)  {};
\node (d2)[label=left:{$d$}] at (9,2)  {};
\node (e2)[label=left:{$e$}] at (11,2)  {};
\node (f2)[label=left:{$f$}] at (13,2)  {};
\node (g2)[label=left:{$g$}] at (9,4)  {};
\node (h2)[label=left:{$h$}] at (11,4)  {};
\node (i2)[label=left:{$i$}] at (13,4)  {};
\end{tikzpicture} 
%\label{fig:move1}
%\caption{How a flip of the tiling affects the corresponding Schr\"oder path}
%\end{figure}
%\begin{figure}
\begin{tikzpicture}[scale = 0.6]
\node[align=left] at (6.5,2) {$\longleftrightarrow$};
\filldraw[fill=blue!40!white, draw=black] (0,2) rectangle (4,4);
\filldraw[fill=blue!40!white, draw=black] (0,0) rectangle (4,2);
\node (b3) at (0,4)[circle,fill=black]  {};
\node (t3) at (4,4)[circle,fill=black]  {};
\draw[-latex] (b3) -> (t3);

\filldraw[fill=blue!40!white, draw=black] (11,0) rectangle (13,4);
\node (b) at (11,2)[circle,fill=black]  {};
\node (t) at (13,4)[circle,fill=black]  {};
\draw[-latex] (b) -> (t);
\filldraw[fill=blue!40!white, draw=black] (9,0) rectangle (11,4);
\node (b22) at (9,4)[circle,fill=black]  {};
\node (t2) at (11,2)[circle,fill=black]  {};
\draw[-latex] (b22) -> (t2);

\node (a1)[label=left:{$a$}] at (0,0)  {};
\node (b1)[label=left:{$b$}] at (2,0)  {};
\node (c1)[label=left:{$c$}] at (4,0)  {};
\node (d1)[label=left:{$d$}] at (0,2)  {};
\node (e1)[label=left:{$e$}] at (2,2)  {};
\node (f1)[label=left:{$f$}] at (4,2)  {};
\node (g1)[label=left:{$g$}] at (0,4)  {};
\node (h1)[label=left:{$h$}] at (2,4)  {};
\node (i1)[label=left:{$i$}] at (4,4)  {};

\node (a2)[label=left:{$a$}] at (9,0)  {};
\node (b2)[label=left:{$b$}] at (11,0)  {};
\node (c2)[label=left:{$c$}] at (13,0)  {};
\node (d2)[label=left:{$d$}] at (9,2)  {};
\node (e2)[label=left:{$e$}] at (11,2)  {};
\node (f2)[label=left:{$f$}] at (13,2)  {};
\node (g2)[label=left:{$g$}] at (9,4)  {};
\node (h2)[label=left:{$h$}] at (11,4)  {};
\node (i2)[label=left:{$i$}] at (13,4)  {};
\end{tikzpicture} 
\caption{A flip of a tiling and the corresponding local move on Schr\"oder paths.
\label{fig:move1}}
\end{figure}
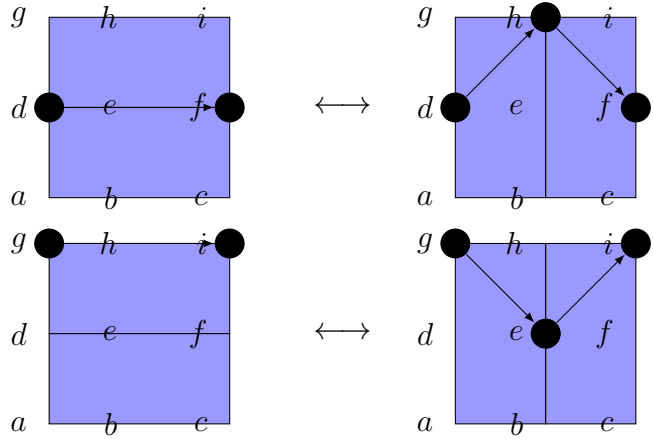
\end{lemma}

\begin{proof}
The statement is clear by inspection for the tilings. Checking the assertion for Schr\"oder paths
is more complicated. We need to examine the various  cases of what the 
path looks like on the left and right of the square being modified.
In other words, we need to specify whether the path increases, stays flat or
decreases as it enters node $d$, and ditto for when it leaves node $f$.
 One such case  is seen in Figure \ref{local_move}.
 When we perform  that local move,
 the   exponents of $b$ and $h$ in \eqref{sweight} increase by $1$,
while the exponents of both $d$ and $f$ decrease by $1$.
 All other cases are similar.
\end{proof}

\begin{figure}[b]
\centering
 \begin{subfigure}[b]{0.45\textwidth}
        \centering
        \includegraphics[width=\textwidth]{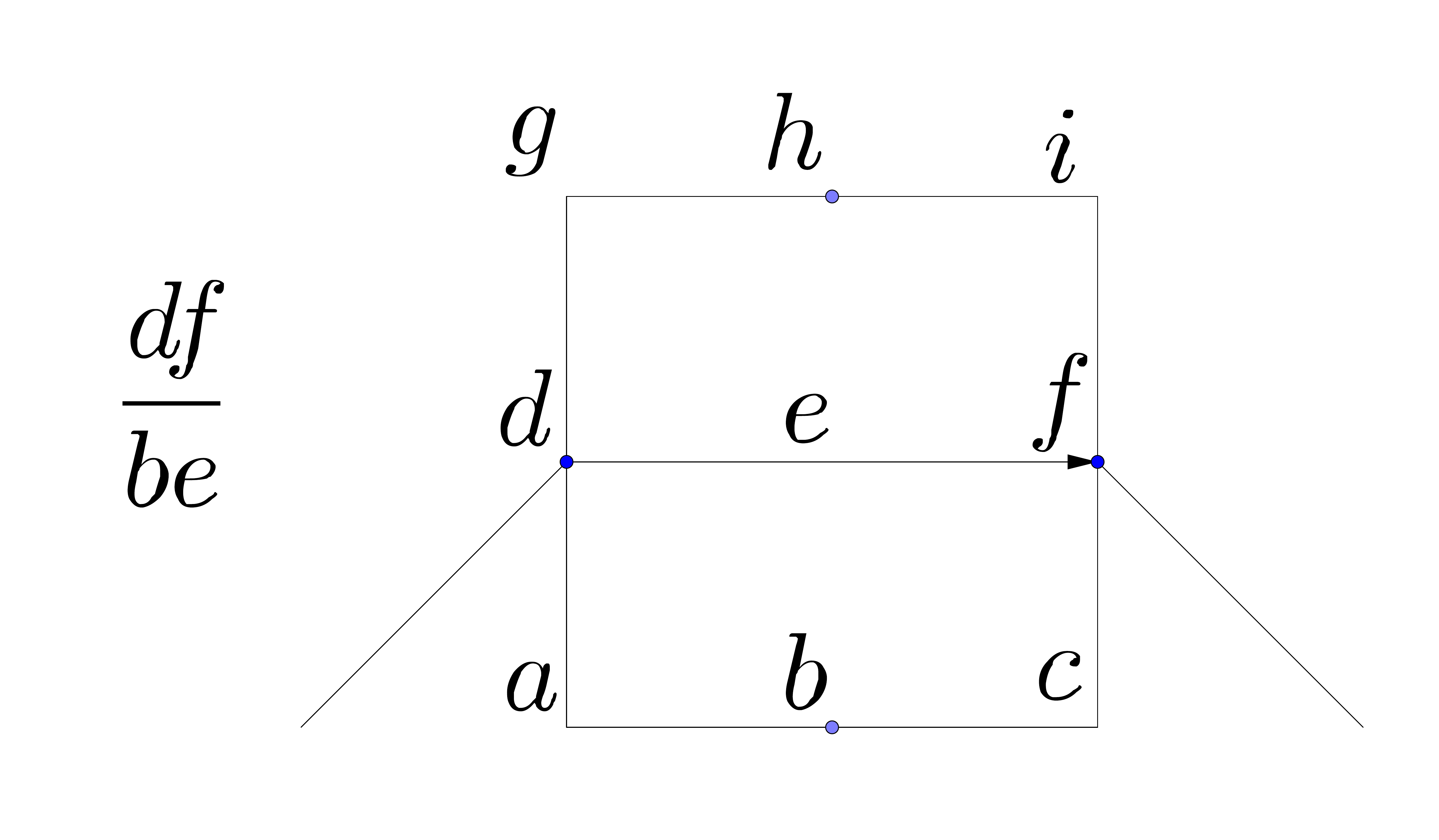}
    \end{subfigure}
    \hfill
    \begin{subfigure}[b]{0.45\textwidth}
        \centering
        \includegraphics[width=\textwidth]{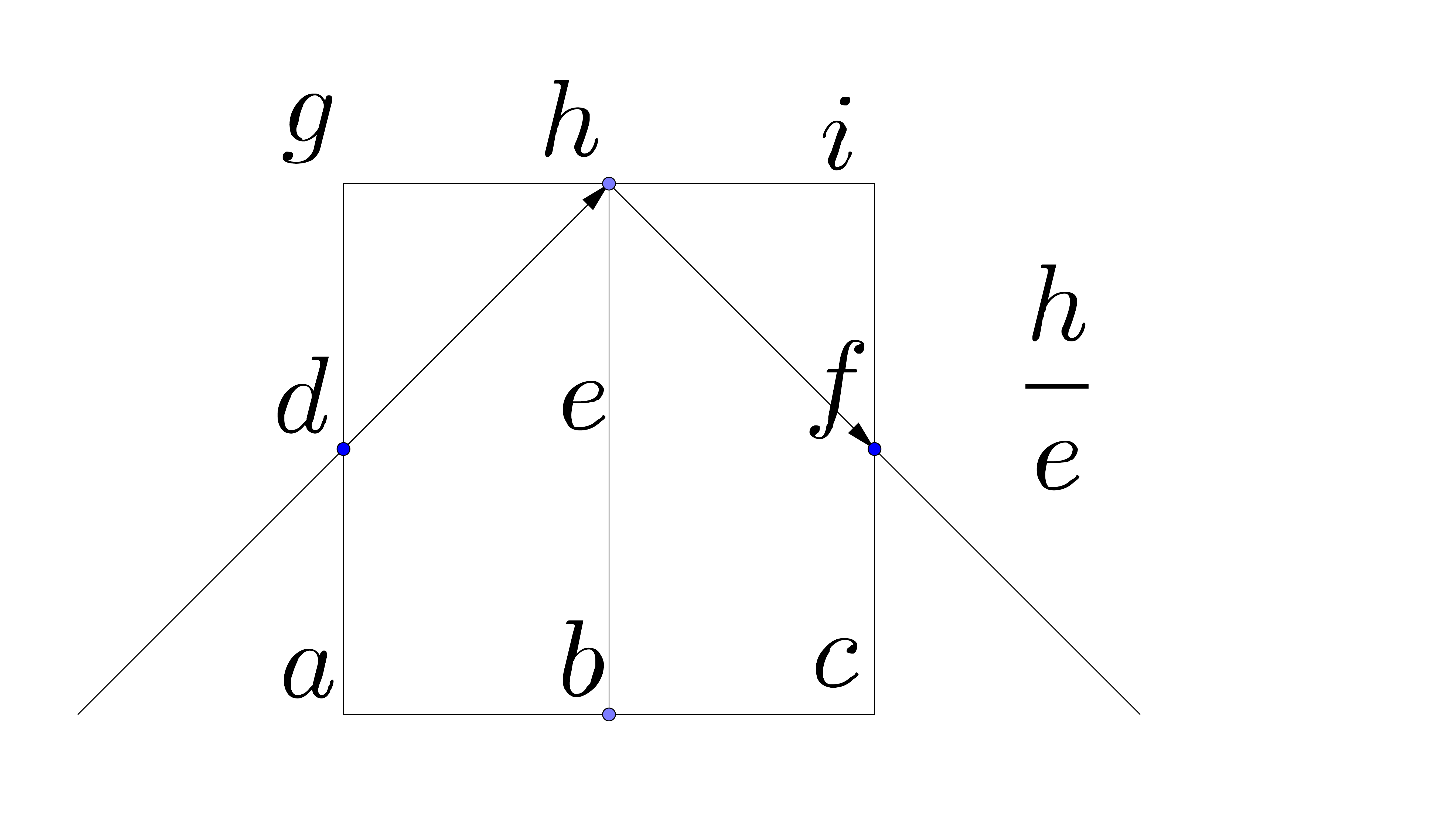}
    \end{subfigure}
    \vspace{-0.2in}
    \caption{\label{local_move}This local move multiplies the weight of the Schr\"{o}der path by 
    $\frac{b h}{d f}$. }
\end{figure}

\begin{proposition}\label{PathWeight}
If $T$ is a tiling in $\A_n(2j,2i-1)$,
where $i > j$, then
$W_\SSS(\Phi(T))=W_\A(T)$.
\end{proposition}

\begin{proof}
It is well known \cite{tilings} that two domino tilings of a simply connected region
can always be connected by a sequence of \emph{flips}, 
where a flip is the local move that 
switches two  horizontal tiles for two vertical tiles or vice-versa, as 
seen in Figure \ref{fig:move1}. 

Let $T_0$ be the tiling consisting only of horizontal tiles.
The corresponding Schr\"oder path
$\Phi(T_0)$ is  a horizontal path.  Here, 
the two objects have the same weight:
\begin{align}\label{horizontal_tiling_weight}
W_\SSS(\Phi(T_0)) \,=\, W_\A(T_0) \,= \,
\frac{a_{j+1,j} a_{j+2,j+1} a_{j+3,j+2} \dots a_{i,i-1}}{p_{j+1} p_{j+2} p_{j+3} \dots p_{i-1}}.
\end{align}
By Lemma \ref{lem:move}, if $W_\SSS(\Phi(T)) = W_\A(T)$ and $T'$ is obtained by a flip,
then $W_\SSS(\Phi(T')) = W_\A(T')$.  
Since the tilings in $\A_n(2j,2i{-}1)$ are connected by flips, the assertion follows.
 \end{proof}
 
 \begin{proof}[Proof of Theorem \ref {thm:Schroder}]
 This follows  from Theorem \ref{KP-theorem}, Lemma \ref{bij} and
 Proposition \ref{PathWeight}.
 \end{proof}

\section{Back to symmetric matrices}\label{sec:proof}

The strategy for proving Theorem \ref{thm:main} is to 
combine Theorem \ref{thm:Schroder}  with a projection
from Schr\"oder paths to Catalan paths. 
Let $S$ be any Schr\"oder path in $G'_n$.  The associated Catalan path
$\pi(S)$ in $G_n$ is defined by
\begin{itemize}
\item replacing each horizontal step in $S$ with a strict local minimum,  i.e.~a 
southeast step followed by a 
northeast step;
\item adding a northeast step at the beginning of $S$
and a southeast step at the end of $S$.
\end{itemize}
If $S$ starts at $i$ and ends at $j-1$ in $G'_n$ then
$\pi(S)$ starts at $i$ and ends at $j$ in $G_n$.
Figure~\ref{fig3:projection} shows how four of the  six Schr\"oder paths in 
$\SSS_4(1,3)$ map to four of the five Catalan paths in 
$\C_4(1,4)$.   The two  other Schr\"oder paths 
in Figure \ref{fig:Schroder2}
 map to the Catalan path in Figure~\ref{fig:Catalan}.

\begin{figure}[h]
\centering
\includegraphics[height=7.2in]{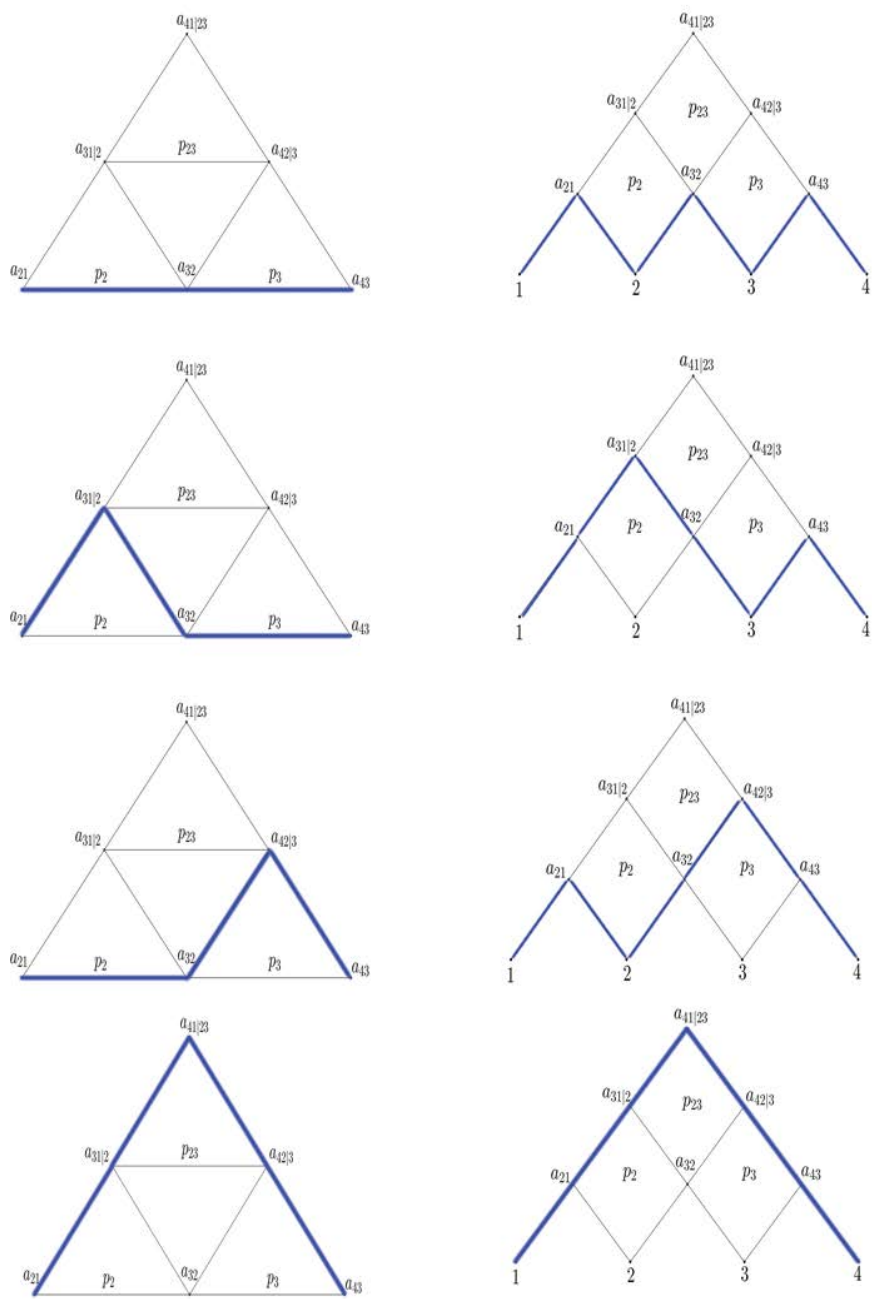}
\vspace{-0.3in}
\caption{The Schr\"oder paths (left) are projected to the Catalan paths (right).
\label{fig3:projection}}
\end{figure}

%\begin{figure}[H]
%\centering
%\includegraphics[height=11cm]{Graphics/aztecWeight}
%\end{figure}

Theorem \ref{KP-theorem} is an immediate consequence of
Theorem \ref{thm:Schroder}  and the following proposition.

\begin{proposition} \label{schroder_to_catalan_sum}
The weight of a Catalan path is the sum of the weights of the Schr\"oder
paths in its preimage under the projection $\pi$, i.e. 
\begin{equation}\label{claim}
\sum_{S \in \pi^{-1}(C)} \!\!  W_{\SSS}(S) \,\,= \,\,W_{\C}(C).
\end{equation}
Here the labels of the paths come from a symmetric matrix, i.e.
$x_{ij} = x_{ji}$ for all $i$ and $j$.
\end{proposition}

The proof will rely on equation \eqref{eq:relation} and Lemma \ref{KPlemma}. 
Using Muir's law of extensible minors,  we obtain  the following identity 
that expresses connected almost-principal minors
of a symmetric $n \times n$ matrix
in terms of connected principal minors:
\begin{equation}
\label{eq:relation} \qquad
a_{ij|I}^2 \,-\, p_I p_{I \cup \{i, j\}} - p_{I \cup \{i\}} p_{I \cup \{j\}} \,\,=\,\,0 , \qquad
2 \leq i < j \leq n{-}1, \,\,
I = \{i{+}1,\ldots,j{-}1\}.
\end{equation}

We now use this identity to prove the following claim.

\begin{lemma}\label{KPlemma}
Let $S'$ and $S$ be two Schr\"oder paths in $\SSS_n$ that
are related as shown in the 
bottom row of Figure \ref{fig:move1} (with $S'$  on
the left and $S$ on the right).  If the 
%shown at the left and right in Figure \ref{fig:move1},
labels come from a symmetric $n \times n$ matrix, then
the resulting weights of these paths satisfy
\begin{equation}
\label{eq:WWW}
W_{\SSS}(S) \,+\, W_{\SSS}(S') \,\,= \,\, \frac{e^2}{bh} W_{\SSS}(S).
\end{equation}
\end{lemma}

\begin{proof}
The label $e$ of the local minimum in $S$ is an almost-principal minor,
while $b,h,d,f$ are principal minors.
By  \eqref{eq:relation}, it satisfies $\,e^2 = bh+df$, and hence
$  \frac{e^2}{bh} = 1 + \frac{df}{bh}$. By Lemma~\ref{lem:move}, we have 
$W_{\SSS}(S') = \frac{df}{bh} W_{\SSS}(S) .$
This implies $\,\frac{e^2}{bh} W_{\SSS}(S)\,=\,
W_{\SSS}(S) + W_{\SSS}(S')$.
\end{proof}

\begin{example} \label{ex:sixtofive} \rm
Let $S'$ and $S$ be the fourth and fifth Schr\"oder paths
in Figure \ref{fig:Schroder2}, with labels 
given by
a symmetric $4 {\times} 4$ matrix. Using the identity
$a_{23} = p_{23} + p_2 p_3$, as in \eqref{eq:relation}, we  find
$$ W_{\SSS}(S) \,+\, W_{\SSS}(S') \,\,= \,\,
\frac{a_{13|2} a_{24|3}}{p_2 p_3 a_{23}} \, + 
\frac{a_{13|2} a_{24|3}}{p_{23} a_{23}} \,\, = \,\,
\frac{a_{13|2} a_{24|3} a_{23}}{p_2p_{23}  p_3 }.
$$
This explains how the
six terms in \eqref{eq:x41} become the five terms of $x_{14}$ shown in \eqref{eq:Xmatrix}.
Namely, the weight of  the Catalan path in  Figure \ref{fig:Catalan} is
 the sum of the fourth and fifth terms in \eqref{eq:x41}.
\end{example}

\begin{proof}[Proof of Proposition \ref{schroder_to_catalan_sum}]
Let $C$ be a Catalan path with
 $m$ local minima. It  has
$m+1$ local maxima.
Let $A_1, \dots,A_{m}$ and $A'_1, \dots, A'_{m+1}$ denote the variables
at the local minima and maxima, respectively.
Let $P_1, \dots, P_{m}$ and $P'_1,\dots, P'_{m+1}$ denote the face variables
located directly above the minima and directly below the maxima, respectively.
Then 
\begin{equation}\label{star}
W_{\C}(C) \,\,= \,\,\frac{A_1 \cdots A_{m} A'_1 \cdots A'_{m+1}}
{P_1\cdots P_{m} P'_1 \cdots  P'_{m+1}}.
\end{equation}
We also denote the face variables located directly below the local minima
by $P''_1 ,\dots ,P''_{m}$.

There are $2^{m}$ Schr\"{o}der paths that project to $C$ via $\pi$. These
correspond to the $2^{m}$ choices of either  preserving a local minimum,
or replacing it by a horizontal edge.
We denote the Schr\"oder paths in $\pi^{-1}(C)$ by 
$S_{d_1 d_2 \dots d_{m}}$, where $d_i = 0$ 
if the local minimum at $A_i$ was preserved and
$d_i = 1$ if it was replaced by a horizontal edge.
By Lemma \ref{KPlemma}, we have 
$$ \begin{matrix}
W_{\SSS}(S_{0d_2 \cdots d_m}) + W_{\SSS}(S_{1 d_2 \cdots d_m})
& =  & \frac{A_1^2}{P_1 P''_1} W_{\SSS}(S_{0d_2 \cdots d_m}), \qquad \\
W_{\SSS}(S_{d_1 0 d_3 \cdots d_m}) + W_{\SSS}(S_{d_1 1 d_3 \cdots d_m})
& = & \frac{A_2^2}{P_2 P''_2} W_{\SSS}(S_{d_1 0 d_3\cdots d_m}), \ldots
\end{matrix}
$$
By aggregating these identities, we obtain 
$$
\sum_{S \in \pi^{-1}(C)} \!\! \, W_{\SSS}(S)  \,\, = \,\,
\sum_{d_1,\ldots,d_m \in \{0,1\}} \!\!\!\! W_{\SSS}( S_{d_1 d_2 \cdots d_m}) \,\,=\,\,
\frac{A_1^2 A_2^2 \cdots A_{m}^2}{P_1 P''_1 P_2 P''_2 \cdots P_m P''_m}
W_{\SSS}(S_{0 0 \dots 0}).$$
But, now it follows from \eqref{sweight} and \eqref{star} that 
$$W_{\SSS}(S_{0 0 \dots 0}) \,\,= \,\,\frac{A'_1 \cdots A'_{m+1} P''_1 \cdots P''_m}
{P'_1\cdots P'_{m+1} A_1 \cdots A_{m}} \,\,\, = \,\,\,
\frac{P_1 P''_1 P_2 P''_2 \cdots P_m P''_m} {A_1^2 A_2^2 \cdots A_{m}^2} W_{\C}(C). 
$$
Therefore 
the sum of the weights of the Schr\"oder paths in $\pi^{-1}(C)$ is equal to 
$W_{\C}(C)$.
\end{proof}

\begin{remark} \label{rem:primeideal} \rm
The expression in Theorem \ref{thm:main} is not the only way to 
express the entries of a symmetric matrix in terms of the 
$\binom{n}{2} + \binom{n-2}{2} + n$
connected almost-principal and principal minors.
The ideal of polynomial relations among these minors is generated by 
the $\binom{n-2}{2}$ quadrics in \eqref{eq:relation}.
Indeed, Theorem~\ref{thm:main} ensures that the algebra generated
by these minors has dimension $\binom{n+1}{2}$, so their relation ideal 
has codimension $\binom{n-2}{2} = \binom{n}{2} + \binom{n-2}{2} + n - \binom{n+1}{2}$.
The $\binom{n-2}{2}$ relations \eqref{eq:relation} lie in that ideal and they generate a
complete intersection. That complete intersection is a prime ideal because
 none of the $a_{ij|I}$ lie in the subalgebra generated by
the  principal minors.
For instance, for $n=4$, our prime ideal  is principal. It is
$\langle a_{23}^2-p_2 p_3-p_{23} \rangle$.
\end{remark}

\section{Parametrizing Correlation Matrices}
\label{sec5}

We now specialize to real symmetric $n \times n$ matrices
that are positive definite and have all diagonal entries equal to $1$.
Such matrices are known as {\em correlation matrices}.
They play an important role in statistics, notably in the study of
multivariate normal distributions.
The set $\mathcal{E}_n$ of all $n \times n$ correlation matrices 
 is an open convex set of dimension $\binom{n}{2}$.
Its closure is a convex body, known in optimization theory \cite{frgbook, LP}
under the name {\em elliptope}.

In certain statistical applications it is desirable to 
generate random correlation matrices. Specifically,
one wishes to sample from the uniform distribution on the
elliptope $\mathcal{E}_n$. A solution to this problem was given
by Joe \cite{MR2301633} and further refined
by  Lewandowski {\it et al.}~\cite{MR2543081}.
The underlying geometric idea is to construct a parametrization from the standard cube:
$$ \Psi: (-1,1)^{\binom{n}{2}} \rightarrow \mathcal{E}_n. $$
The papers \cite{MR2301633, MR2543081} describe such 
maps $\Psi$ that are algebraic and bijective, so they identify the open cube
with the open elliptope. However, the construction is recursive.
In what follows we revisit the formula in \cite{MR2301633}
and we make it completely explicit.  Remarkably, it is precisely the
restriction of our Laurent polynomial parametrization
in Theorem \ref{thm:main} to the region where
all connected principal minors $p_I$ are positive and $p_1 = \cdots = p_n = 1$.

Let $X = (x_{ij})$ be a real symmetric $n \times n$ matrix.
We assume that $X$ is positive definite, i.e.~all 
principal minors $p_I$ are strictly positive. In statistics,
such an $X$ serves as the covariance matrix
of a normal distribution on $\R^n$, whose
 {\em partial correlations} are given by
 \begin{equation}
 \label{eq:rho} \qquad
\rho_{ij|I}  \,\, = \,\,  \frac{(-1)^{\lceil\, |I|/2 \,\rceil} \cdot a_{ij|I}}{\sqrt{p_{iI} \cdot p_{jI}}}
\qquad
\hbox{where} \,\,\, i,j \not \in I \,\, \hbox{and} \,\, i < j .
\end{equation}
For $I = \emptyset$, we obtain the $\binom{n}{2}$ entries
of the correlation matrix $Y = (y_{ij})$, namely
$$ y_{ij} \,=\,\rho_{ij} \,=\, \frac{a_{ij}}{\sqrt{ p_i p_j}} \, = \,
\frac{x_{ij}}{\sqrt{x_{ii} x_{jj}}} \qquad \hbox{for} \,\,1 \leq i < j \leq n.$$
The partial correlation  $\rho_{ij|I} $ in \eqref{eq:rho} is called {\em connected} if 
$I = \{i{+}1,i{+}2,\dots,j{-}2,j{-}1\}$.

\begin{theorem} \label{thm:correl}
The $\binom{n}{2}$ entries $y_{ij}$ of a correlation matrix 
can be written uniquely in terms of the $\binom{n}{2}$ connected
partial correlations $\rho_{ij|I}$.
Explicit formulas are derived from those in Theorem \ref{thm:main}
by first replacing each occurrence of a parameter $\,a_{ij|I}\,$ by 
$\, (-1)^{\lceil\, |I|/2 \,\rceil}\rho_{ij|I} \sqrt{p_{iI}p_{jI}}\,$ and thereafter  replacing each occurrence of a
 parameter $\,p_{r,r+1,\ldots,s}\,$ by the 
product of the $\binom{s-r+1}{2}$ expressions 
$\,(-1)^{\lfloor\, |I|/2\, \rfloor} (1-\rho_{ij|I}^2)\,$ where 
$r {\leq} i {<} j {\leq} s$ and $\,I = \{i{+}1,i{+}2,\ldots,j{-}1\}$.
The resulting map
$\Psi: (\rho_{ij|I}) \mapsto (y_{ij})$
is a bjection between $(-1,1)^{\binom{n}{2}}$ and $\,\mathcal{E}_n$.
\end{theorem}

\begin{proof}
The replacement formula for $a_{ij|I}$ is
seen in \eqref{eq:rho}. The formula for the signed
principal minors
$\,p_{r,r+1,\ldots,s}\,$ in terms of connected partial correlations
is due to Joe  \cite[Theorem 1]{MR2301633}.
It can be derived by recursively applying the following version of 
\eqref{eq:relation} in concert with \eqref{eq:rho}:
\begin{align} \label{HermitianRelation}
p_{ijI} \,\,=\,\, \frac{a_{ij|I}^2-p_{iI}p_{jI}}{p_I}, \quad I=\{i+1,\ldots,j-1\}.
\end{align}
Our formulas give an algebraic map
$\Psi: (\rho_{ij|I}) \mapsto (y_{ij})$
between affine spaces of dimension $\binom{n}{2}$.
This map is invertible on  $\mathcal{E}_n$ because each partial correlation
$\rho_{ij|I}$ can be written via \eqref{eq:rho}
in terms of the entries $y_{ij}$ of the correlation matrix.
All partial correlations are real numbers strictly between $-1$ and $1$.
The connected partial correlations $\rho_{ij|I}$
 can vary freely, as explained in \cite[page 2179]{MR2301633}.
 From this, we get the desired bijection.
\end{proof}

We now illustrate our parametrization of correlation matrices
in the two smallest cases.

\begin{example}[$n=3$] \rm
We consider the open $3$-dimensional cube defined by the inequalities
$$ -1 \,\,< \,\, \rho_{12}, \,\rho_{23},\, \rho_{13|2} \,\,< \,\,1. $$
Our bijection $\Psi$ identifies each  point in this cube with a
 $3 \times 3$ correlation matrix:
$$ \!
 \begin{bmatrix}
1 & y_{12} & y_{13} \\
y_{12} & 1 & y_{23} \\
y_{13} & y_{23} & 1 \end{bmatrix}  = \,
\begin{bmatrix}
1 & \rho_{12} & \rho_{12} \rho_{23} - \rho_{13|2} (1{-} \rho_{12}^2)^{\frac{1}{2}}
(1{-}\rho_{23}^2)^{\frac{1}{2}} \\
\,\rho_{12} \,& 1 & \rho_{23} \\
 \rho_{12} \rho_{23} - \rho_{13|2} (1{-} \rho_{12}^2)^{\frac{1}{2}} (1{-}\rho_{23}^2)^{\frac{1}{2}} 
& \rho_{23} & 1 
\end{bmatrix}.
$$
One checks that this matrix is positive definite, and,
as in \cite[Theorem 1]{MR2301633},
its determinant 
$$ {\rm det}(Y) \,\, = \,\,
(1- \rho_{12}^2)(1- \rho_{23}^2)(1-\rho_{13|2}^2) $$
defines the facets of the cube.
It is instructive to draw how the boundary of the cube maps onto
 the boundary of the elliptope $\mathcal{E}_3$.
The latter is depicted in \cite[Figure 5.8, page 232]{frgbook}.
\end{example}

The combinatorics of our planar graph $G_n$ and its Catalan paths
can be seen in a different guise in  \cite{MR2301633, MR2543081}.
These correspond to the structures called {\em D-vines} in these papers.
Figure \ref{fig:Dvine} shows the standard D-vine for $n=4$. Its edges
are naturally labeled with the six coordinates of the cube, namely
$\rho_{12},\rho_{23},\rho_{34},\rho_{13|2},\rho_{24|3},\rho_{14|23}$.
These correspond to the six almost-principal minors 
$a_{ij|I}$ in the labeled
graph $G_4$ in Figure \ref{fig:Catalan}.

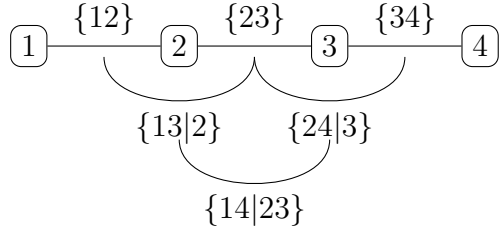
\begin{figure}[H]
\begin{tikzpicture}
\node (1) at (0,2)[rectangle,rounded corners,draw,align=center]  {$1$};
\node (2) at (2,2)[rectangle,rounded corners,draw,align=center]  {$2$};
\node (3) at (4,2)[rectangle,rounded corners,draw,align=center]  {$3$};
\node (4) at (6,2)[rectangle,rounded corners,draw,align=center]  {$4$};
\node (5) at (2,0.9)[rectangle] {};
\node (6) at (4,0.9)[rectangle]  {};
\node (7) at (1,2)[rectangle] {};
\node (8) at (3,2)[rectangle]  {};
\node (9) at (5,2)[rectangle]  {};

\path[]
(1) edge node[above] {$\{12\}$} (2)
(2) edge node[above] {$\{23\}$} (3)
(3) edge node[above] {$\{34\}$} (4)
(7) edge [bend right=90]  node[below] {$\{13|2\}$} (8)
(8) edge [bend right=90]  node[below] {$\{24|3\}$} (9)
(5) edge [bend right=90]  node[below] {$\{14|23\}$} (6);
\end{tikzpicture} 
\caption{\label{fig:Dvine}
The standard D-vine for four random variables.}
\end{figure}

\begin{example}[$n=4$] \rm
The $4 {\times} 4$ correlation matrix $Y$ is obtained from the matrix $X$
in~\eqref{eq:Xmatrix} by  
performing the replacements that are described in Theorem \ref{thm:correl}. 
We first substitute
$$ 
\begin{matrix}
 a_{12} = \rho_{12} \sqrt{p_1 p_2}, & 
 a_{23} = \rho_{23} \sqrt{p_2 p_3}, &
 a_{34} = \rho_{34} \sqrt{p_3 p_4}, \\
a_{13|2}= - \rho_{13|2} \sqrt{p_{12} p_{23}}, & 
a_{24|3} = - \rho_{24|3} \sqrt{p_{23} p_{34}}, &
 a_{14|23}= - \rho_{14|23} \sqrt{p_{123} p_{234}},
 \end{matrix}
$$
and then we eliminate the connected principal minors as follows:
$$ \begin{matrix}
 p_1 = 1,\,
  p_2 = 1,\,
  p_3 = 1,\,
  p_4 = 1,\,
 p_{12} = -(1-\rho_{12}^2),\,
 p_{23} = -(1-\rho_{23}^2), \,
 p_{34} = -(1-\rho_{34}^2), \\
p_{123} = -(1-\rho_{12}^2)(1-\rho_{23}^2)(1-\rho_{13|2}^2) \quad {\rm and} \quad
p_{234} = -(1-\rho_{23}^2)(1-\rho_{34}^2)(1-\rho_{24|3}^2).
\end{matrix}
$$
This results in the formulas for the six entries of $Y$ in terms of
$\rho_{12},\rho_{23},\rho_{34},\rho_{13|2},\rho_{24|3},\rho_{14|23}$.
These give the bijection $\Psi$ between the cube and
the elliptope, both of dimension six.
It is instructive to verify that ${\rm det}(Y)$ is the product of
the facet-defining expressions $(1-\rho_\bullet)^2$.
\end{example}

The paper \cite{MR2543081} argues that C-vines are better than 
D-vines when it comes to sampling from the elliptope $\mathcal{E}_n$.
It would be interesting to examine both C-vines and D-vines
from the network perspective of \cite{KP2} and to explore
whether Catalan-type formulas for $X$ can be derived from these as well.
Could such vines play a role in the theory of cluster algebras?

\bigskip

\noindent {\bf Acknowledgements}.
We thank Richard Kenyon, Robin Pemantle and Caroline Uhler
for helpful conversations. This project was supported
by the National Science Foundation 
(DMS-1419018, DGE-1106400, DMS-1049513).

\bigskip

\bibliographystyle{alpha}

\begin{thebibliography}{}

\end{thebibliography}


\begin{thebibliography}{LKJ09}

\bibitem{frgbook}
Grigoriy~Blekherman, Pablo~Parrilo and Rekha~Thomas:
\newblock {\em Semidefinite Optimization and Convex Algebraic Geometry},
\newblock MOS-SIAM Ser. Optim., 13, SIAM, Philadelphia, PA, 2013. 

\bibitem{MR2301633}
Harry Joe:
\newblock Generating random correlation matrices based on partial correlations,
\newblock {\em J. Multivariate Anal.}, 97(10):2177--2189, 2006.

\bibitem{KP2}
Richard Kenyon and Robin Pemantle:
\newblock Principal minors and rhombus tilings,
\newblock {\em J. Phys. A: Math. Theor.}, 47:474010, 2014.


\bibitem{LP} Monique Laurent and Svatopluk Poljak:
\newblock   On the facial structure of the set of correlation matrices,
\newblock {\em SIAM J. Matrix Anal. Appl.} 17: 530--547, 1996.
 
\bibitem{MR2543081}
Daniel Lewandowski, Dorota Kurowicka, and Harry Joe:
\newblock Generating random correlation matrices based on vines and extended
  onion method,
\newblock {\em J. Multivariate Anal.}, 100(9):1989--2001, 2009.

\bibitem{tilings}
Nicolau Saldanha and Carlos Tomei:
\newblock An overview of domino and lozenge tilings,
\newblock {\em Resenhas IME-USP}, 2(2):239--252, 1995.

\bibitem{Stanley}
Richard~P. Stanley:
\newblock {\em Enumerative Combinatorics, {V}ol. 2}, volume~62 of {\em
  Cambridge Studies in Advanced Mathematics},
\newblock Cambridge University Press, 1999.

\end{thebibliography}

\bigskip
\bigskip

\footnotesize 
\noindent {\bf Authors' address:}

\smallskip

Department of Mathematics, University of California, 
Berkeley, CA 94720-3840, USA

{\tt $\{$bernd,e.tsukerman,williams$\}$@math.berkeley.edu}

\end{document}